\documentclass[12pt]{amsart}
\usepackage[left=2.5cm,right=2.5cm, top=2.5cm, bottom=2.5cm]{geometry}
\usepackage{hyperref}
\usepackage[utf8]{inputenc}
\usepackage[english]{babel}
\usepackage{multirow}
\usepackage{verbatim}
\usepackage{color}
\title{Tropical Gaussians: a brief survey}
\author{Ngoc Mai Tran}
\address{Department of Mathematics, University of Texas at Austin, Texas TX 78712, USA and Hausdorff Center for Mathematics, Bonn 53115, Germany}
\email{ntran@math.utexas.edu}
\thanks{The author would like to thank Bernd Sturmfels for raising the question that inspired this paper. The author would also like to thank Yue Ren and Martin Ulirsch, the organizers of the Tropical Panorama conference at the Max-Planck Institute for Mathematics in the Sciences, for the opportunity to speak about this work while it was still in progress.}
\date{\today}

\DeclareMathOperator{\R}{\mathbb{R}}

\renewcommand{\P}{\mathbb{P}}
\newcommand{\TP}{\mathbb{TP}}

  \theoremstyle{plain}
\newtheorem{theorem}{Theorem}[section]
\newtheorem{lemma}[theorem]{Lemma}

\newtheorem{proposition}[theorem]{Proposition}
  \theoremstyle{definition}

\DeclareMathOperator{\val}{val}

\begin{document}

\maketitle

\begin{abstract}
We review the existing analogues of the Gaussian measure in the tropical semiring and outline various research directions.
\end{abstract}

\section{Introduction}
\label{sec:intro1}

Tropical mathematics has found many applications in both pure and applied areas, as documented by a growing number of monographs on its interactions with various other areas of mathematics: algebraic geometry \cite{baker2016nonarchimedean,gross2011tropical,huh2016tropical,maclagan2015introduction}, 
discrete event systems \cite{baccelli1992synchronization, butkovivc2010max}, large deviations and calculus of variations \cite{kolokoltsov1997idempotent,puhalskii2001large}, and 
combinatorial optimization \cite{joswig2014essentials}. 
At the same time, new applications are emerging in phylogenetics \cite{lin2018tropical,yoshida2017tropical,page2020tropical}, statistics \cite{hook2017max}, economics \cite{baldwin2013tropical,crowell2016tropical,elsner2004max,gursoy2013analytic,joswig2017cayley,shiozawa2015international,tran2013pairwise,tran2015product}, game theory, and complexity theory \cite{allamigeon2018log,akian2012tropical}. There is a growing need for a systematic study of probability distributions in tropical settings. Over the classical algebra, the Gaussian measure is arguably the most important distribution to both theoretical probability and applied statistics. 
In this work, we review the existing analogues of the Gaussian measure in the tropical semiring. We focus on the three main characterizations of the classical Gaussians central to statistics: invariance under orthonormal transformations, independence and orthogonality, and stability. We show that some notions do not yield satisfactory generalizations, others yield the classical geometric or exponential distributions, while yet others yield completely different distributions. There is no single notion of a `tropical Gaussian measure' that would satisfy multiple tropical analogues of the different characterizations of the classical Gaussians. This is somewhat expected, for the interaction between geometry and algebra over the tropical semiring is rather different from that over $\R$. Different branches of tropical mathematics lead to different notions of a tropical Gaussian, and it is a worthy goal to fully explore all the options. We conclude with various research directions.  

\section{Three Characterizations of the classical Gaussian}
\label{sec:intro}

The Gaussian measure $\mathcal{N}(\mu,\Sigma)$, also called the normal distribution with mean $\mu \in \R^n$ and covariance $\Sigma \in \R^{n \times n}$ is the probability distribution with density
$$ f_{\Sigma,\mu}(x) \propto \exp(-\frac{1}{2}(x-\mu)^\top \Sigma^{-1} (x-\mu)), \quad x \in \R^n. $$
Let $\mathbf{I}$ denote the identity matrix, and $\mathbf{0} \in \R^n$ the zero vector. Measures $\mathcal{N}(\mathbf{0},\Sigma)$ are called centered Gaussians, while $\mathcal{N}(\mathbf{0},\mathbf{I})$ is the standard Gaussian.
Any Gaussian can be standardized by an affine linear transformation. 
\begin{lemma}\label{lem:standard}
Let $\Sigma = U\Lambda U^\top$ be the eigendecomposition of $\Sigma$. Then 
$X \sim \mathcal{N}(\mu,\Sigma)$ if and only if $(U\Lambda^{1/2})^{-1}(X-\mu) \sim \mathcal{N}(\mathbf{0},\mathbf{I})$.
\end{lemma}
The standard Gaussian has two important properties. First, if $X$ is a standard Gaussian in $\R^n$, then its coordinates $X_1, \dots, X_n$ are $n$ independent and identically distributed (i.i.d) random variables. Second,
for any orthonormal matrix $A$, $AX \stackrel{d}{=} X$. These two properties completely characterize the standard Gaussian \cite[Proposition 11.2]{kallenberg2006foundations}. This result was first formalized in dimension three by Maxwell \cite{maxwell1860v} when he studied the distribution of gas particles, though the essence of his argument was made by Herschel \cite{herschel1850quetelet} ten years earlier, as pointed out in \cite[p10]{bryc2012normal}. 

\begin{theorem}[Maxwell] \label{thm:maxwell}
Let $X_1, \dots, X_n$ be i.i.d univariate random variables, where $n \geq 2$. Then the distribution of $X = (X_1, \dots, X_n)$ is spherically symmetric iff the $X_i$'s are centered Gaussians on $\R$.
\end{theorem}

From a statistical perspective, Lemma \ref{lem:standard} and Theorem \ref{thm:maxwell} reduces working with data from the Gaussian measure to doing linear algebra. In particular, if data points come from a Gaussian measure, then they are the affine linear transformation of data points from a standard Gaussian, whose coordinates are always independent regardless of the orthonormal basis that it is represented in. These properties are fundamental to Principal Component Analysis, an important statistical technique whose tropical analogue is actively being studied \cite{yoshida2017tropical}.

There are numerous other characterizations of the Gaussian measure whose ingredients are only orthogonality and independence, see 
\cite[\S 1.9]{bogachev1998gaussian} and references therein. One famous example is Kac's theorem \cite{kac1939characterization}. It is a special case of the Darmois-Skitovich theorem \cite{darmois1953analyse,skitovitch1953property}, which characterizes Gaussians (not necessarily centered) in terms of independence of linear combinations. A multivariate version of this theorem is also known, see \cite{kagan1973characterization}. 

\begin{theorem}[Darmois-Skitovich]
Let $X_1, \dots, X_n$ be independent univariate random variables. Then the $X_i$'s are Gaussians if and only if there exist $\alpha,\beta \in \R^n$, $\alpha_i,\beta_i \neq 0$ for all $i = 1, \dots, n$, such that $\sum_i\alpha_iX_i$ and $\sum_i\beta_iX_i$ are independent.
\end{theorem}

Another reason for the wide applicability of Gaussians in statistics is the Central Limit Theorem. An interesting historical account of its development can be found in \cite[\S 4]{kallenberg2006foundations}. From the Central Limit Theorem, one can derive yet other characterizations of the Gaussian, such as the distribution which maximizes entropy subject to a fixed variance \cite{barron1986entropy}. The appearance of the Gaussian in the Central Limit Theorem is fundamentally linked to its characterization as the unique $2$-stable distribution. This is expressed in the following theorem by P\'olya \cite{polya1923herleitung}. There are a number of variants of this theorem, see \cite{bogachev1998gaussian,bryc2012normal} and discussions therein.
\begin{theorem}[P\'olya]
Suppose $X,Y \in \R^n$ are independent random variables. Then $X,Y$ and $(X+Y)/\sqrt{2}$ have the same distribution iff this distribution is the centered Gaussian.
\end{theorem}

%TODO: add back the thing about the heat equation here?

\section{Tropical analogues of Gaussians}

\subsection{Tropicalizations of $p$-adic Gaussians}\label{sec:local.fields}
Evans \cite{evans2001local} used Kac's Theorem as the definition of Gaussians to extend them to local fields. Local fields are finite algebraic extensions of either the field of $p$-adic numbers or the field of formal Laurent series with coefficients drawn from the finite field with $p$ elements \cite{evans2001local}. In particular, local fields come with a tropical valuation $\val$, and thus one can define a tropical Gaussian to be the tropicalization of the Gaussian measure on a local field. A direct translation of \cite[Theorem 4.2]{evans2001local} shows that the tropicalization of the one-dimensional $p$-adic Gaussian is the classical geometric distribution. 

\begin{proposition}[Tropicalization of the $p$-adic Gaussian]\label{lem:p-adic}
For a prime $p \in \mathbb{N}$, let $X$ be a $\mathbb{Q}_p$-valued Gaussian with index $k \in \mathbb{Z}$. Then $\val(X)$ is a random variable supported on $\{k, k+1, k+2, \dots\}$, and it is distributed as $k + \emph{geometric}(1-p^{-1})$. That is, 
$$ \mathbb{P}(\val(X) = k+s) = p^{-s}(1-p^{-1}) \mbox{ for } s = 0, 1, 2, \dots. $$
\end{proposition}
\begin{proof}
Recall that a non-zero rational number $r \in \mathbb{Q} \backslash \{0\}$ can be uniquely written as $r = p^s(a/b)$ where $a$ and $b$ are not divisible by $p$, in which case the valuation of $r$ is $|r| := p^{-s}$. The completion of $\mathbb{Q}$ under the metric $(x,y) \mapsto |x-y|$ is the field of $p$-adic numbers, denoted $\mathbb{Q}_p$. The tropical valuation of $r$ is $\val(r) := s$. By \cite[Theorem 4.2]{evans2001local}, the family of $\mathbb{Q}_p$-valued Gaussians is indexed by $\mathbb{Z}$. For each $k \in \mathbb{Z}$, there is a unique $\mathbb{Q}_p$-valued Gaussian supported on the ball $p^k\mathbb{Z}_p := \{x \in \mathbb{Q}_p: |x| \leq p^{-k}\}$. Furthermore, the Gaussian is the normalized Haar measure on this support. As $p^k\mathbb{Z}_p$ is made up of $p$ translated copies of $p^{k+1}\mathbb{Z}_p$, which in turn is made up of $p$ translated copies of $p^{k+2}\mathbb{Z}_p$, a direct computation yields the density of $\val(X)$. 
\end{proof}

There is a large and growing literature surrounding probability on local fields, or more generally, analysis on ultrametric spaces. They have found diverse applications, from spin glasses, protein dynamics, and genetics, to cryptography and geology; see the recent comprehensive review \cite{dragovich2017p} and references therein. The $p$-adic Gaussian was originally defined as a step towards building Brownian motions on $\mathbb{Q}_p$ \cite{evans2001local}. It would be interesting to use tools from tropical algebraic geometry to revisit and expand results involving random $p$-adic polynomials, such as the expected number of zeroes in a random $p$-adic polynomial system \cite{MR2266718}, or properties of determinants of matrices with i.i.d $p$-adic Gaussians \cite{evans2002elementary}. Previous work on random $p$-adic polynomials from a tropical perspective tends to consider systems with uniform valuations \cite{avendano2011multivariate}. Lemma \ref{lem:p-adic} hints that to connect the two literatures, the geometric distribution may be more suitable.

\subsection{Gaussians via tropical linear algebra}
Consider arithmetic done in the tropical algebra $(\overline\R,\oplus,\odot)$, where $\overline{\R}$ is $\R$ together with the additive identity. In the max-plus algebra 
$(\overline\R,\overline{\oplus},\odot)$ where $a \overline{\oplus} b = \max(a,b)$, for instance, $\overline{\R} = \R \cup \{-\infty\}$. In the min-plus algebra $(\overline\R,\underline{\oplus},\odot)$ where $a \underline{\oplus} b = \min(a,b)$, we have $\overline{\R} = \R \cup \{+\infty\}$. 
To avoid unnecessary technical details, in this section we focus on vectors taking values in $\R$ instead of $\overline{\R}$. %The inclusion of the additive identity does not impact our discussions. 

Tropical linear algebra was developed by several communities with different motivations. It evolved as a linearization tool for certain problems in discrete event systems, queueing theory and combinatorial optimization; see the monographs \cite{baccelli1992synchronization,butkovivc2010max}, as well as the recent survey \cite{komenda2018max} and references therein.
A large body of work focuses on using the tropical setting to find analogous versions of classical results in linear algebra and convex geometry. Many fundamental concepts have rich tropical analogues, including the spectral theory of matrices \cite{akian2006max,baccelli1992synchronization,butkovivc2010max}, linear independence and projectors \cite{AlGK09,akian2011best,butkovivc2007generators,sergeev2009multiorder}, separation and duality theorems in convex analysis \cite{briec2008halfspaces,cohen2004duality,gaubert2011minimal,nitica2007max}, matrix identities \cite{gaubert1996burnside,hollings2012tropical,morrison2016tropical,simon1994semigroups}, matrix rank \cite{chan20114,develin2005rank,izhakian2009tropical,shitov2011example}, and tensors \cite{butkovic2018tropical,tsukerman2015tropical}. 
Another research direction focuses on the combinatorics of
objects arising in tropical convex geometry, such as polyhedra and hyperplane arrangements \cite{akian2012tropical,develin2004tropical,joswig2016weighted,joswig2007affine,sturmfels2012combinatorial,tran2017enumerating}. 
These works have close connections to matroid theory and are at the interface of tropical linear algebra and tropical algebraic geometry \cite{ardila2009tropical,fink2015stiefel,giansiracusa2017grassmann,hampe2015tropical,loho2018matching}.  

Despite the rich theory of tropical linear algebra, in this section we shall show that there is currently no satisfactory way to define the tropical Gaussian as a classical probability measure based on the characterizations of Gaussians via orthogonality and independence as in Section \ref{sec:intro}. This is somewhat surprising, for there are good analogues of norms and orthogonal decomposition in the tropical algebra. In hindsight, the main difficulty stems from the fact that such tropical analogues are compatible with tropical arithmetic, while classical measure theory was developed with the usual algebra. In Section \ref{sec:k} we consider the idempotent probability measure theory, where there is a well-defined Gaussian measure complete with a quadratic density function analogous to the classical case.

The natural definition for tropical linear combinations of $v_1,\dots,v_m \in \R^n$ is the set of vectors of the form
\begin{equation}\label{[v]}
[v_1,\dots,v_m] := \{a_1 \odot v_1 \oplus \dots \oplus a_m \odot v_m \mbox{ for } a_1, \dots, a_m \in \R\},
\end{equation}
where scalar-vector multiplication is defined pointwise. That is, for $a \in \R$ and $v \in \R^n$, $a \odot v \in \R^n$ is the vector with entries
$$ (a \odot v)_i = a + v_i \mbox{ for } i = 1,\dots,n.  $$
We shall also write $a + v$ for $a \odot v$, with the convention scalar-vector addition is defined pointwise.

For finite $m$, $V := [v_1,\dots,v_m]$ is always a compact set in $\mathbb{TP}^{n-1} := \R^n / \R \mathbf{1}$ \cite{develin2004tropical}. Unfortunately, this means one cannot hope to have finitely many vectors to `tropically span' $\R^m$. Nonetheless, there is a well-defined analogue orthogonal projection in the tropical algebra.  Associated to a tropical polytope $V := [v_1,\dots,v_m]$ defined by \eqref{[v]} is the canonical projector $P_V: \R^n \to V$ that plays the role of the orthogonal projection onto $V$ \cite{cohen2004duality}. This projection is compatible with the projective Hilbert metric $d_H$ \cite{cohen2000hahn,cohen2004duality}, in the sense that $P_V(x)$ is a best-approximation under the projective Hilbert metric of $x$ by points in $V$ \cite{cohen2004duality,akian2011best}. When $V$ is a polytrope, that is, a tropical polytope that is also classically convex, then $P_V$ can be written as a tropical matrix-vector multiplication. This is analogous to classical linear algebra, where best-approximations in the Euclidean distance can be written as a matrix-vector multiplication. 

In the max-plus algebra, the projective Hilbert metric is defined by
$$ d_H(x,y) = \max_{i,j\in[n]}(x_i-y_i+y_j-x_j). $$
It induces the Hilbert projective norm $\|\cdot\|_H: \R^m \to \R$ via $\|x\|_H = d_H(x,0).$ Since
$d_H(x,y) = \max_i(x_i-y_i) - \min_j(x_j-y_j)$, %= \|(x-y) - \min_j(x_j-y_j)\|_\infty, $ %% \max_i(x_i-y_i) + \max_j(y_j-x_j). $$
one finds that
$$ \|x\|_H = \|x - \min_i x_i\|_\infty. $$
This formulation shows that the projective Hilbert norm plays the role of the $\ell_\infty$-norm on~$\TP^{n-1}$. The appearance of $\ell_\infty$, instead of $\ell_2$, agrees with the conventional `wisdom' that generally in the tropical algebra, $\ell_2$ is replaced by $\ell_\infty$ \cite{evans2001local}. 

To generalize Maxwell's characterization of the classical Gaussians, we need a concept of orthogonality. One could attempt to mimic orthogonality via the orthogonal decomposition theorem, as done in \cite{evans2001local} for the case of local fields discussed in Section \ref{sec:local.fields}. Namely, over a normed space $(\mathcal{Y},\|\cdot\|)$ over some field $K$, say that $y_1, \dots, y_m \in \mathcal{Y}$ are orthogonal if and only if for all $\alpha_i \in K$, the norm of the vector $\sum_i\alpha_iy_i$ equals the norm of the vector $(|\alpha_1|\|y_1\|,\dots,|\alpha_m|\|y_m\|)$, that is,
\begin{equation}\label{eqn:orthogonal}
\| \sum_i\alpha_iy_i \| = \|(|\alpha_1|\|y_1\|,\dots,|\alpha_m|\|y_m\|)\|.
\end{equation}
 In the Euclidean case, this is the Pythagorean identity
$$ \|\sum_i\alpha_iy_i\|_2 = (\sum_i|\alpha_i|^2\|y_i\|^2)^{1/2}, $$
for example. The $\ell_\infty$-norm, unfortunately, does not work well with the usual notion of independence in probability. In the Hilbert projective norm, (\ref{eqn:orthogonal}) can be interpreted either as
\begin{equation}\label{eqn:first}
\|\max_i(\alpha_i + y_i)\|_H = \max_i\|\alpha_i+y_i\|_H - \min_i\|\alpha_i+y_i\|_H  = \max_i\|y_i\|_H - \min_i\|y_i\|_H
\end{equation}
or
\begin{equation}\label{eqn:second}
\|\max_i(\alpha_i + y_i)\|_H = \max_i(\alpha_i + \|y_i\|_H) - \min_i (\alpha_i + \|y_i\|_H).
\end{equation}
Unfortunately, neither formulation give a satisfactory notion of orthogonality. In \eqref{eqn:first}, as the norm is projective, the coefficients $\alpha_i$ have disappeared from the RHS. This does not support the notion that over an orthogonal set of vectors in the classical sense, computing the norm of linear combinations is the same as computing norm of the vector of coefficients. In \eqref{eqn:second}, for sufficiently large $\alpha_1$, the RHS increases without bound whereas the LHS is bounded, and thus equality cannot hold for all $\alpha_i \in \R$ over \emph{any} generating set of $y_i$'s. 

The Darmois-Skitovich characterization for Gaussians also does not generalize well. Note that the additive identity in $(\overline{\R},\oplus,\odot)$ is either $-\infty$ or $+\infty$, so the condition that $\alpha_i,\beta_i \neq 0$ becomes redundant. The following lemma states that the any compact distribution will satisfy the Darmois-Skitovich condition.
\begin{lemma}
Let $X_1,\dots,X_n$ be independent random variables on $\R^n$. Then there exist $\alpha,\beta \in \R^n$ such that $\displaystyle\bigoplus_{i=1}^n\alpha_i\odot X_i$ and $\displaystyle\bigoplus_{i=1}^n\beta_i\odot X_i$ are independent if and only if $X_1, \dots, X_n$ have compact support.
\end{lemma}
\begin{proof}
Let us sketch the proof for $n = 2$ under the min-plus algebra. Let $X = (X_1,X_2) \in \R^2$ and $Y = (Y_1,Y_2) \in \R^2$ be two independent variables. Define $\overline{F}_X, \overline{F}_Y: \R^2 \to [0,1]$ via $\overline{F}_X(t) = \P(X \geq t)$ and $\overline{F}_Y(t) = \P(Y \geq t).$
Fix $\alpha,\beta \in \R^2$. For $t \in \R^2$, 
\begin{align*}
\P(\alpha_1\odot X \oplus \alpha_2 \odot Y \geq t) &= \P(\min(\alpha_1+X,\alpha_2+Y) \geq t) \mbox{ by definition } \\
%=& \P(\alpha_1+X_1 \geq t_1, \alpha_2+Y_1 \geq t_1, \alpha_1+X_2 \geq t_2, \alpha_2+Y_2 \geq t_2) \\
&= \P(X \geq t-\alpha_1)\P(Y \geq t-\alpha_2) \mbox{ by independence} \\
&= \overline{F}_X(t-\alpha_1)\overline{F}_Y(t-\alpha_2).
\end{align*}
Meanwhile, 
\begin{align*}
& \P(\alpha_1\odot X \oplus \alpha_2 \odot Y \geq t, \beta_1\odot X \oplus \beta_2 \odot Y \geq t) \\ 
=&\, \P(\min(\alpha_1+X,\alpha_2+Y) \geq t, \min(\beta_1+X,\beta_2+Y) \geq t) \mbox{ by definition} \\
%=& \P(\alpha_1+X_1 \geq t_1, \alpha_2+Y_1 \geq t_1, \alpha_1+X_2 \geq t_2, \alpha_2+Y_2 \geq t_2) \\
=&\, \P(X \geq t-\alpha_1, X \geq t-\beta_1)\P(Y \geq t-\alpha_2, Y \geq t-\beta_2) \mbox{ by independence} \\
=& \, \min(\overline{F}_X(t-\alpha_1), \overline{F}_X(t-\beta_1)) \min(\overline{F}_Y(t-\alpha_2), \overline{F}_Y(t-\beta_2)).
\end{align*}
Therefore, for $\alpha_1\odot X \oplus \alpha_2 \odot Y$ and $\beta_1\odot X \oplus \beta_2 \odot Y$ to be independent, for all $t \in \R^2$, we need
\begin{align*} %TODO: format. 
&\overline{F}_X(t-\alpha_1)\overline{F}_X(t-\beta_1)\overline{F}_Y(t-\alpha_2)\overline{F}_Y(t-\beta_2) \\
=& \min(\overline{F}_X(t-\alpha_1), \overline{F}_X(t-\beta_1)) \min(\overline{F}_Y(t-\alpha_2), \overline{F}_Y(t-\beta_2))\cdot \\
&\max(\overline{F}_X(t-\alpha_1), \overline{F}_X(t-\beta_1)) \max(\overline{F}_Y(t-\alpha_2), \overline{F}_Y(t-\beta_2))\\
=& \min(\overline{F}_X(t-\alpha_1), \overline{F}_X(t-\beta_1)) \min(\overline{F}_Y(t-\alpha_2), \overline{F}_Y(t-\beta_2)). 
\end{align*}
But $\overline{F}_X$ and $\overline{F}_Y$ are non-increasing functions taking values between $0$ and $1$. So 
$$\overline{F}_X(t-\alpha_1)\overline{F}_X(t-\beta_1)\overline{F}_Y(t-\alpha_2)\overline{F}_Y(t-\beta_2) \leq \min(\overline{F}_X(t-\alpha_1), \overline{F}_X(t-\beta_1)) \min(\overline{F}_Y(t-\alpha_2), \overline{F}_Y(t-\beta_2)),$$ 
and equality holds if and only if 
%$$\max(\overline{F}_X(t-\alpha_1),\overline{F}_X(t-\beta_1)) = \max(\overline{F}_Y(t-\alpha_2),\overline{F}_Y(t-\beta_2)) = 1,$$
$$ \min(\overline{F}_X(t-\alpha_1), \overline{F}_X(t-\beta_1)) = 0, \mbox{ or } \min(\overline{F}_Y(t-\alpha_2), \overline{F}_Y(t-\beta_2)) = 0. $$
As either of these scenarios must hold for each $t \in \R^2$, we conclude that $X$ and $Y$ must have compact supports. 
Conversely, suppose that $X$ and $Y$ have compact supports. 
Then one can choose $\alpha_1 = \beta_2 = 0$ and $\alpha_2 = \beta_1$ be a sufficiently large number, so that 
$$ \alpha_1\odot X \oplus \alpha_2 \odot Y = X, \mbox{ and } \beta_1\odot X \oplus \beta_2 \odot Y = Y. $$
In this case, the Darmois-Skitovich condition holds trivially, as desired.
\end{proof}

Now consider Polya's condition. Here the Gaussian is characterized via stability under addition. When addition is replaced by minimum, it is well-known that this leads to the classical exponential distribution. One such characterization, which generalizes to distributions on arbitrary lattices, is the following \cite[Theorem 3.4.1]{bryc2012normal}.

\begin{theorem}
Suppose $X,Y$ are independent and identically distributed nonnegative random variables. Then this distribution is the exponential if and only if for all $a,b > 0$ such that $a+b=1$, $\min(X/a,Y/b)$ has the same distribution as $X$.
\end{theorem}

By considering $\log(X)$ and $\log(Y)$, one could restate this theorem in terms of the min-plus algebra, though the condition $a+b=1$ does not have an obvious tropical interpretation. This shows that the tropical analogue of Gaussian is the classical exponential distribution. 

\subsection{Gaussians in idempotent probability}\label{sec:k}
Idempotent probability is a branch of idempotent analysis, which is functional analysis over idempotent semirings \cite{kolokoltsov1997idempotent}. Idempotent semirings are characterized by the additive operation being idempotent, that is, $a \oplus a = a$. The tropical semirings used in the previous sections are idempotent, but there are others, such as the Boolean semiring in semigroup theory. Idempotent analysis was developed by Litvinov, Maslov and Shipz
\cite{litvinov1998linear} in relation to problems of calculus of variations. Closely related are the work on large deviations \cite{puhalskii2001large}, which has found applications in queueing theory, as well as fuzzy measure theory and logic \cite{dubois2012fundamentals,wang2013fuzzy}. 
The work we discussed in this section is based on
that of Akian, Quadrat and Viot and co-authors  \cite{akian2011best,akian1994bellman}, whose goal was to develop idempotent probability as a theory completely in parallel to classical probability. Following their convention, we work over the min-plus algebra. 

All fundamental concepts of probability have an idempotent analogue, see \cite{akian1994bellman} and references therein. For a flavor of this theory, we compare the concept of a measure. In classical settings, a probability measure $\mu$ is a map from the $\sigma$-algebra on a space $\Omega$ to $\R_{\geq 0}$ that satisfies three properties: (i) $\mu(\emptyset) = 0$, (ii) $\mu(\Omega) = 1$, and (iii) for a countable sequence $(E_i)$ of pairwise disjoint sets,
$$ \mu(\bigcup_{i=1}^\infty E_i) = \sum_{i=1}^\infty\mu(E_i). $$
The analogous object in the min-plus probability is the cost measure $\mathbb{K}$ defined by three axioms:
(i) $\mathbb{K}(\emptyset) = +\infty$, (ii) $\mathbb{K}(\Omega) = 0$, and 
$$ \mathbb{K}(\bigcup_iE_i) = \bigoplus_i \mathbb{K}(E_i) = \inf_i \mathbb{K}(E_i). $$
Idempotent probability is rich and has interesting connections with dynamic programming and optimization. For instance, tropical matrix-vector multiplication can be interpreted as an update step in a Markov chain, so the Bellman equation plays the analogue of the Kolmogorov-Chapman equation. Most notably, the classical quadratic form~$(x-y)^2/2\sigma^2$ defines a stable distribution \cite{akian1994bellman}. Furthermore, it is the unique density that is invariant under the Legendre-Fenchel transform \cite{akian1994bellman}, which is the tropical analogue of the Fourier transform \cite{kolokoltsov1997idempotent}. This is in parallel to the characterization of the scaled version of the Gaussian density $x \mapsto \exp(-\pi x^2)$ being invariant under the Fourier transform. While $x \mapsto \exp(-\pi x^2)$ is not the unique function to possess this property \cite{duffin1948function}, the fact that the Fourier transform of a $\mathcal{N}(\mu,\sigma^2)$ univariate Gaussian has the form $x \mapsto \exp(i\mu x - \frac{x^2}{2})$ is frequently employed to prove independence of linear combinations of Gaussians. Under this light, one can regard the idempotent measure correspond to the density $(x-y)^2/2\sigma^2$ to be the idempotent analogue of the classical Gaussian. 

\section{Open directions} 

\subsection{Tropical curves, metric graphs and Gaussians via the Laplacian operator}

From the perspective of stochastic analysis, the Gaussian measure can be characterized as the unique invariant measure for the Ornstein-Uhlenbeck semigroup \cite[\S 1]{bogachev1998gaussian}. This semigroup is a powerful tool in proving hypercontractivity and log-Sobolev inequalities. In particular, the Gaussian density can be characterized as the function that satisfies such inequalities with the best constants \cite{bogachev1998gaussian}. One useful characterization of the Ornstein-Uhlenbeck semigroup is by its generator, whose definition has two ingredients: a Laplacian operator and a gradient operator $\nabla$. Let us elaborate. Let $\gamma$ be a centered Gaussian measure on $\R^n$. The Ornstein-Uhlenbeck semigroup $(T_t,t\geq 0)$ is defined on $L^2(\gamma)$ by the Mehler formula
$$ T_th(x) = \int_{\R^n} h(e^{-t}x + \sqrt{1-e^{-2t}}y)\, \gamma(dy), \quad t > 0 $$
and $T_0$ is the identity operator. It characterizes the Gaussian measure in the following sense \cite[\S 1]{bogachev1998gaussian}. 
\begin{lemma}
$\gamma$ is the unique invariant probability measure for $(T_t,t \geq 0)$.
\end{lemma}
One can arrive at this semigroup without the Mehler's formula as follows. Let $\mathcal{D} = \{h \in L^2(\gamma): \lim_{t \to 0} \frac{T_t h - h}{t} \mbox{ exists in the norm of } L^2(\gamma)\}$. (Recall that $L^2(\gamma)$ is the space of square integrable functions with respect to the measure $\gamma$).
The linear operator $L$ defined on $\mathcal{D}$ by
$$ Lh = \lim_{t \to 0} \frac{T_t h - h}{t} $$
is called the generator of the semigroup $(T_t, t \geq 0)$. The generator of the Ornstein-Uhlenbeck semigroup is given by
$$ Lh (x) = \Delta h (x) - \langle x, \nabla h \rangle = \sum_{i=1}^n \frac{\partial^2 h}{\partial x_i^2}(x) - \sum_{i=1}^n x_i \frac{\partial h}{\partial x_i}(x). $$
This generator uniquely specifies the semigroup. Importantly, the two ingredients needed to define $L$ are the Laplacian operator $\Delta$, and the gradient operator $\nabla$. Thus the semigroup can be defined on Riemannian manifolds, for instance. This opens up ways to define Gaussians on tropical curves. 

In tropical algebraic geometry, an abstract tropical curve is a metric graph \cite{mikhalkin2008tropical}. There are some minor variants: with vertex weights \cite{brannetti2011tropical,chan2012tropical}, or just the compact part \cite{baker2006metrized}. An embedded tropical curve is a balanced weighted one-dimensional complex in $\R^n$. There are several constructions of tropical curves. In particular, they arise as limits of amoebas through a process called Maslov dequantization in idempotent analysis \cite{litvinov1998linear}. Tropical algebraic geometry took off with the landmark paper of Mikhalkin \cite{mikhalkin2005enumerative}, who used tropical curves to compute Gromov-Witten invariants of the plane $\mathbb{P}^2$ \cite{maclagan2015introduction}. Since then, tropical curves, and more generally, tropical varieties, have been studied in connection to mirror and symplectic geometry \cite{gross2011tropical}. Another heavily explored aspect of tropical curves is their divisors and Riemann-Roch theory \cite{baker2007riemann,baker2016nonarchimedean,gathmann2008riemann,mikhalkin2008tropical}. 
This theory is connected to chip-firing and sandpiles, which were initially conceived as deterministic models of random walks on graphs \cite{cooper2006simulating}. 

Metric graphs are Riemannian manifolds with singularities \cite{baker2006metrized}. Brownian motions defined on metric graphs, heat semigroups on graphs, and graph Laplacians are an active research area \cite{kostrykin2012brownian,post2008spectral}. As of now, however, the author is unaware of an analogue of the Ornstein-Uhlenbeck semigroup and its invariant measure on graphs. It would also be interesting to study what Brownian motion on graphs reveals about tropical curves and their Jacobians.

\subsection{Further open directions}

The natural ambient space for doing tropical convex geometry is not $\R^m$, but $\mathbb{TP}^{n-1}$, where a vector $x \in \R^m$ is identified with all of its scalar multiples $a \odot x$. Probability theory on classical projective spaces relies on group representation \cite{benoist2014random}. Unfortunately, there is no satisfactory tropical analogue of the general linear group. Every invertible $n \times n$ matrix with entries in $\overline{\R}$ is the composition of a diagonal matrix and a permutation of the standard basis of $\overline{\R}^n$ \cite{kolokoltsov1997idempotent}. We note that several authors have studied tropicalization of special linear group over a field with valuation \cite{joswig2007affine,werner2011tropical}.
It would be interesting to see whether this can be utilized to define probability measures on $\TP^{n-1}$. 

Another approach is to `fix' the main difficulty with the idempotent algebra, namely, the lack of the additive inverse. Some authors have put back the additive inverse and developed a theory of linear algebra in this new algebra, called the supertropical algebra \cite{izhakian2008supertropical}. It would be interesting to study matrix groups and their actions under this algebra, and in particular, pursue the definition of Gaussians as invariant measures under actions of the orthogonal group.  

\subsection{Beyond Gaussians}

In a more applied direction, $\TP^{n-1}$ is a natural ambient space to study problems in economics, network flow and phylogenetics. Thus one may want an axiomatic approach to finding distributions on $\TP^{n-1}$ tailored for specific applications. For instance, in shape-constrained density estimation, log-concave multivariate totally positive of ordered two (MTP2) distributions are those whose density $f: \R^d \to \R$ is log-concave and satisfies the inequality
$$ f(x)f(y) \leq f(x \vee y)f(x \wedge y) \mbox{ for all } x,y \in \R^d. $$
A variety of distributions belong to this family. Requiring that such inequalities hold for all $x,y \in \TP^{n-1}$ leads to the stronger condition of $L^\natural$-concavity
$$ f(x)f(y) \leq f((x+\alpha\mathbf{1}) \vee y)f(x \wedge (y-\alpha \mathbf{1})) \mbox{ for all } x,y \in \R^d, \alpha \geq 0. $$
A Gaussian distribution is log-concave MTP2 if and only if the inverse of its covariance matrix is an $M$-matrix \cite{lauritzen2019maximum}. Only diagonally dominant Gaussians are $L^\natural$-concave \cite[\S 2]{murota2003discrete}. This subclass of densities has nice properties that make them algorithmically tractable in Gaussian graphical models \cite{walk_summable,LBP_inverse_M}. In particular, density estimation for $L^\natural$-concave distributions is significantly easier than for log-concave MTP2 \cite{robeva2018maximum}. It would be interesting to pursue this direction to define distributions on the space of phylogenetic trees.

\bibliographystyle{alpha}
\bibliography{survey}

\end{document}